\documentclass[final,12pt,a4paper]{amsart}
\usepackage{graphicx}
\usepackage[all]{xy}
\usepackage{enumitem}
\usepackage{amssymb}
\usepackage{latexsym}
\usepackage{amsmath}
\usepackage{mathrsfs}
\usepackage{array,booktabs}

\usepackage[notref, notcite]{showkeys}  % for draft versions only
\usepackage{xcolor}
\usepackage{hyperref}
\hypersetup{
    colorlinks,
    linkcolor={red!50!black},
    citecolor={blue!50!black},
    urlcolor={blue!80!black}
}

\newcommand{\harxiv}[1]{\href{http://arxiv.org/abs/#1}{\texttt{arXiv:#1}}}
\newcommand{\hyref}[2]{\hyperref[#2]{#1~\ref*{#2}}}
\newcommand{\hypageref}[1]{\hyperref[#1]{page~\pageref*{#1}}}

\newcommand{\coloneqq}{\mathrel{\mathop:}=}

\usepackage{fullpage}

\def\clap#1{\hbox to 0pt{\hss#1\hss}}

\def\mathrlap{\mathpalette\mathrlapinternal}

\def\mathrlapinternal#1#2{%
\rlap{$\mathsurround=0pt#1{#2}$}}

\theoremstyle{plain}
\newtheorem{theorem}{Theorem}[section]
\newtheorem*{theorem*}{Theorem}
\newtheorem{conjecture}[theorem]{Conjecture}
\newtheorem{lemma}[theorem]{Lemma}
\newtheorem{corollary}[theorem]{Corollary}
\newtheorem{proposition}[theorem]{Proposition}

\theoremstyle{definition}
\newtheorem{remark}[theorem]{Remark}
\newtheorem{example}[theorem]{Example}
\newtheorem*{naive-algorithm}{Na\"ive algorithm}
\newtheorem*{refined-algorithm}{Refined algorithm}
\newtheorem{definition}[theorem]{Definition}
\newtheorem{question}[theorem]{Question}

\newcommand{\Step}[1]{\medskip\noindent\emph{Step #1:}}

\newcommand{\LLambda}{\Lambda(r,n,m)}

\newcommand{\Obj}{\text{Ob}}
\newcommand{\Subs}{\text{Sub}}
\newcommand{\Fact}{\text{Fac}}

\newcommand{\sphc}[1]{\catT_{\!#1}}
\newcommand{\orbit}{\kc_{\!\scalebox{0.6}{\text{ADE}}}}
\newcommand{\yes}{\multicolumn{1}{c}{$\checkmark$}}
\newcommand{\yeb}[1]{\multicolumn{1}{c}{$\phantom{{}_{\!\!#1}} \checkmark_{\!\!#1}$}}
\newcommand{\no} {\multicolumn{1}{c}{$\times$}}
\newcommand{\na} {\multicolumn{1}{c}{$-$}}
\newcommand{\mcc}[1]{\multicolumn{1}{c}{\scalebox{0.9}{#1}}}
\newcommand{\IFq}{\IF_{\!q}}

\newcommand{\dimv}[1]{\underline{\dim}(#1)}
\newcommand{\length}[1]{\text{length}(#1)}

\newcommand{\catT}{\mathsf T}   % for a category called T

\newcommand{\Db}{\sD^b}
\newcommand{\Kb}{\sK^b}
\newcommand{\Kbminus}{\sK^{b,-}}

\newcommand{\begintabularhammock}{ \smallskip\noindent\hspace{0.05\textwidth}
                                   \begin{tabular}{@{} p{0.15\textwidth} @{} p{0.80\textwidth} @{}} }

\newcommand{\IQ}{{\mathbb{Q}}}
\newcommand{\IF}{{\mathbb{F}}}

\newcommand{\IZ}{{\mathbb{Z}}}

\newcommand{\IN}{{\mathbb{N}}}

\newcommand{\kc}{{\mathcal C}}

\newcommand{\supp}{\text{supp}}

\newcommand{\kk}{{\mathbf{k}}}

\DeclareMathOperator{\coker}{\mathsf{coker}}
\DeclareMathOperator{\kernel}{\mathsf{ker}}
\DeclareMathOperator{\image}{\mathsf{im}}

\newcommand{\ind}[1]{\mathsf{ind}(#1)}
\renewcommand{\mod}[1]{\mathsf{mod}(#1)}
\newcommand{\Mod}[1]{\mathsf{Mod}(#1)}
\newcommand{\stmod}[1]{\underline{\mathsf{mod}}(#1)}
\newcommand{\proj}[1]{\mathsf{proj}(#1)}
\newcommand{\Proj}[1]{\mathsf{Proj}(#1)}
\DeclareMathOperator{\add}{\mathsf{add}}

\newcommand{\thick}[2]{\mathsf{thick}_{#1}(#2)}

\DeclareMathOperator{\Hom}{\mathrm{Hom}}
\newcommand{\floor}[1]{\lfloor #1\rfloor}

\newcommand{\sC}{\mathsf{C}}
\newcommand{\sD}{\mathsf{D}}

\newcommand{\sH}{\mathsf{H}}

\newcommand{\sK}{\mathsf{K}}

\newcommand{\sX}{\mathsf{X}}
\newcommand{\sY}{\mathsf{Y}}
\newcommand{\sZ}{\mathsf{Z}}

\DeclareMathAlphabet{\mathpzc}{OT1}{pzc}{m}{it}

\newcommand{\into}{\hookrightarrow}
\newcommand{\onto}{\twoheadrightarrow}
\newcommand{\arrd}{ \ar@{-}[r] \ar@{=}[d] }
\newcommand{\too}{\longrightarrow}
\renewcommand{\iff}{~\Longleftrightarrow~}

\newcommand{\tri}[3]{#1\rightarrow #2\rightarrow #3\rightarrow \Sigma #1}

\renewcommand{\phi}{\varphi}
\renewcommand{\epsilon}{\varepsilon}

\newcommand{\arr}{\ar@{~}[r]}
\newcommand{\arrr}{\ar@{~}[rr]}

\newcommand{\arudd}{\ar[ur] \ar@{.}[dr]}
\newcommand{\aruu}{\ar@{.>}[uuurrr]}
\newcommand{\arurr}{\ar@{.>}[ur] \ar@{.}[rr]}

\newcommand{\bib}[6]{{\bibitem{#2} #3: {\emph{#4},} #5#6.}}

\begin{document}

\title{Discrete triangulated categories}

\author{Nathan Broomhead}
\author{David Pauksztello}
\author{David Ploog}

\begin{abstract}
We introduce and study several homological notions which generalise the discrete derived categories of D. Vossieck. As an application, we show that Vossieck discrete algebras have this property with respect to all bounded t-structures. We give many examples of triangulated categories regarding these notions.
\end{abstract}

\begingroup\renewcommand\thefootnote{}%
\footnote{MSC 2010: 18E30, 16G10, 16E35}
% 16E20: Associative algebras: Grothendieck groups
% 16E35: Associative algebras: Homological methods: Derived categories
% 16G10: Associative algebras: Representations theory of rings and algebras: Representations of Artinian rings
% 16G20: Associative algebras: Representations theory of rings and algebras: Representations of quivers
% 18E30: Category theory: Abelian category: Derived categories, triangulated categories
% 18F30: Category theory: Abelian category: Grothendieck groups
%
\addtocounter{footnote}{-2}\endgroup

\maketitle

\vspace{-2ex}

{\small
\setcounter{tocdepth}{1}
\tableofcontents
}

\vspace{-4ex}

\addtocontents{toc}{\protect{\setcounter{tocdepth}{-1}}}  % No toc entry for Introduction

\section*{Introduction}
\addtocontents{toc}{\protect{\setcounter{tocdepth}{1}}}   % but enable toc entries for other sections

\noindent
In this article, we investigate Hom-finite triangulated categories which are, in various senses, small. Our motivating examples are bounded derived categories of derived-discrete algebras, which were introduced and classified by D. Vossieck \cite{Vossieck}. In previous work \cite{BPP}, we observed some special properties of these categories: the dimension of Hom spaces between indecomposables is bounded (by 1 or 2, depending on the algebra), and all hearts of bounded t-structures have only finitely many indecomposable objects.

We set out to introduce and compare abstract notions which apply to such examples. The three most relevant for this article are
\begin{itemize}
\item \emph{cone finite}: any two objects admit only finitely many cones, up to isomorphism;
\item \emph{hom bounded}: universal bound on Hom dimension among indecomposable objects;
\item \emph{countable}: the category has only countably many objects, up to isomorphism.
\end{itemize}
We establish the following relations between these properties in \hyref{Theorem}{thm:general-properties}:
\begin{theorem*}
\begin{enumerate}[label=\textnormal{(\roman*)}]
\item A hom $1$-bounded triangulated category is cone finite.
\item A cone finite triangulated category with a classical generator is countable.
\end{enumerate}
\end{theorem*}

\noindent 
We give examples showing that countability doesn't imply cone finite or hom bounded, and that cone finite doesn't imply hom bounded.
Vossieck's definition of discreteness does not generalise to abstract triangulated categories, as it invokes cohomology objects, i.e.\ a t-structure. Therefore, we study pairs $(\sD,\sH)$ of a triangulated category together with the heart of a bounded t-structure. Moreover, the Grothendieck group $K_0(\sD)=K_0(\sH)$ enters, generalising dimension vectors of modules. We introduce two notions characterising different aspects of smallness for abelian categories via their K-groups which we call \emph{modular} and \emph{abelian discrete} (see \hyref{Section}{sec:Vossieck}), and we prove in \hyref{Theorem}{thm:discrete-characterisations}:
\begin{theorem*}
Let $(\sD,\sH)$ be a triangulated category with the heart of a bounded t-structure.
\begin{enumerate}[label=\textnormal{(\roman*)}]
\item $\sD$ cone finite and $\sH$ abelian discrete $\implies$ $\sD$ is discrete with respect to $\sH$.
\item $\sD$ is discrete with respect to $\sH$ $\implies$ $\sH$ abelian discrete.
\item $\sD$ is discrete with respect to $\sH$ and $\sH$ modular $\implies$ $\sD$ cone finite.
\end{enumerate}
\end{theorem*}
As one application of this result, we show that derived-discrete algebras are discrete with respect to any bounded t-structure, not just the standard one; see \hyref{Proposition}{prop:DDC-heart-independence}.

In \hyref{Section}{sec:co-discrete}, we introduce and study an analogous definition of discreteness with respect to a bounded co-t-structure with co-heart $\sC$, assuming the existence of a silting object.

There are many interesting examples of triangulated categories having certain of the properties in question. Here, we list some of them. For a more elaborate version, with further properties and example classes, see the table on \hypageref{tab:examples}.
Below, $\Gamma_{\!\text{ADE}}$, $A_\infty$ and $\tilde A_1$ denote any ADE quiver, the one-sided infinite quiver of type $A$ with zigzag orientation, and the Kronecker quiver, respectively. The column then refers to the bounded derived category of the path algebra.
DDC stands for the bounded derived category of a derived-discrete algebra $\LLambda$ with $r < n$, i.e.\ of finite global dimension.
Finally, $\sphc{w}$ is the triangulated category generated by a $w$-spherical object.
\begin{center}
\begin{tabular}{ >{\small} l <{\normalsize} *{8}{p{2.25em}} } \toprule
     & & & & & & & \\[-2.7ex]
     & \mcc{$\kk\Gamma_{\!\scalebox{0.6}{\text{ADE}}}$} & \mcc{$\kk A_\infty$} & \mcc{DDC}
     & \mcc{$\sphc{1}$} & \mcc{$\sphc{>1}$} & \mcc{$\sphc{<1}$}
     & \mcc{$\IQ\tilde A_1$} & \mcc{$\IFq\tilde A_1$} \\ \midrule
%
% examples:            ADE  A_infty   DDC   T_1    T_2    T_0    Q      IF_q
%                                                              
$\sH$-discrete    & \yes & \yes & \yes & \yes & \yes & \na  & \no  & \yes \\
$\sC$-discrete & \yes & \yes & \yes & \na  & \na  & \yes & \no  & \yes \\
hom bounded          & \yes & \yes & \yes & \no  & \yes & \yes & \no  & \no  \\
cone finite          & \yes & \yes & \yes & \yes & \yes & \yes & \no  & \yes \\
countable objects    & \yes & \yes & \yes & \yes & \yes & \yes & \yes & \yes \\
finite hearts        & \yes & \no  & \yes & \no  & \yes & \na  & \no  & \no  \\
\bottomrule
\end{tabular}
\end{center}
\medskip
We mention other approaches for capturing the smallness of categories and algebras:

Small Krull--Gabriel (KG) dimensions of (the abelianisations of) triangulated categories correspond to ``small'' categories. For example, by work of G. Bobi\'nski and H. Krause \cite{BK}, the KG dimension of (perfect categories) of Dynkin quivers is 0, and that of derived-discrete algebras is 1 if the algebra has infinite global dimension and 2 otherwise. However, there are non-derived-discrete algebras whose perfect categories have Krull--Gabriel dimension $2$; for the example of the Kronecker quiver, see \cite[Proposition~1.8]{Kr}.

Another abstract concept for triangulated categories is that of a generic object, and its absence, generic triviality.
In \cite{Han}, Z. Han shows that generic triviality of a compactly generated triangulated category is equivalent to local finiteness of the compact subcategory, and to the compact subcategory having KG dimension 0. Thus it seems unlikely that these notions are useful in the study of the smallness notions investigated here.

In \cite{Adachi-Mizuno-Yang,PSZ} the authors study discreteness of triangulated categories related to finiteness of intervals of t-structures and silting objects. They apply this to contractibility of spaces of stability conditions.

In \cite{Han-Zhang}, Y. Han and C. Zhang characterise derived-discrete algebras as the finite-dimensional algebras of finite global cohomological length. Their approach depends on cohomology and modules, i.e.\ does not apply to abstract triangulated categories.

\subsection*{Acknowledgments:}
It is a pleasure to thank
Martin Kalck,                   % discussion, suggested 'countable' as a good notion
Henning Krause,                 % feedback on preprint
Greg Stevenson                  % feedback on preprint
for their input and suggestions,
as well as an anonymous referee.
%
% Manchester University and grants
Moreover, we thank Mike Prest and The University of Manchester for their hospitality. We are grateful to the London Mathematical Society for financial support via their `Research in Pairs' Scheme 4 grant, no.\ 41434.            % for grant to visit in Manchester
The second named author was supported by
EPSRC grant no.\ EP/K022490/1.   % required on DavidA's papers.

\section{Properties of triangulated categories: cone finite, hom bounded}
\noindent
We define a number of properties that suitable $\kk$-linear triangulated categories can enjoy, all of which capture certain aspects of `smallness'. Throughout, we assume that the class of objects of any category forms a set. Moreover, we apply the following abuse of terminology: whenever we speak of a `set' of objects defined by some property, we mean the class of such objects, up to isomorphism.

Fix a field $\kk$. A $\kk$-linear category is called \emph{Hom-finite} if all homomorphism spaces are finite-dimensional over $\kk$. An additive category is called \emph{Krull--Schmidt} if each object has a decomposition into a finite direct sum of indecomposable objects. The decomposition is unique up to isomorphism and re-ordering of the summands.
Examples are Hom-finite abelian categories and their bounded derived categories.

Throughout this note we shall write $\hom(A,B) = \dim \Hom(A,B)$.

\begin{definition} \label{def:triangulated-notions}
Let $\sD$ be a Hom-finite, Krull--Schmidt $\kk$-linear triangulated category.
\begin{enumerate}[label=(\arabic*)]
\item $\sD$ is called \emph{cone finite} if for any two objects $D_1,D_2\in\sD$, the set of cones of morphisms $D_1 \to D_2$,
      i.e.\ the set $\{C\in\sD \mid \exists \ D_1 \to D_2 \to C \to \Sigma D_1 \}$, is finite.
\item $\sD$ is called \emph{hom $b$-bounded} for some $b\in\IN$ if $\hom(D_1,D_2)\leq b$ for any indecomposable objects $D_1,D_2\in\ind{\sD}$. The minimal such $b$ is called the \emph{hom bound} of $\sD$, and $\sD$ is called \emph{hom bounded} if it is hom $b$-bounded for some $b\in\IN$.
\item $\sD$ is called \emph{countable}, if the set of all objects up to isomorphism is countable. Equivalently, $\ind{\sD}$ is a countable set, as $\sD$ is Krull--Schmidt and Hom-finite.
\end{enumerate}
\end{definition}

All of these definitions could be stated in greater generality: cone finiteness makes sense for all triangulated categories (no field needed); hom boundedness applies to arbitrary $\kk$-linear categories (no triangulated structure required); countability of objects applies to arbitrary categories. The latter notion is crude, and depends strongly on the cardinality of the field $\kk$; see \hyref{Remark}{rem:fields}. We will not explore these properties beyond the setting of Hom-finite triangulated categories.

In this article, we also study the relationship with Vossieck's notion of discreteness, see \hyref{Section}{sec:Vossieck}, and we introduce and investigate its co-t-structure counterpart in \hyref{Section}{sec:co-discrete}.
For now, we only deal with the above three notions: they have the advantage of applying in a general setting, i.e.\ without additional data. Also note each condition (countable objects, hom bounded, cone finite) is automatically passed on to
triangulated subcategories.

\begin{theorem} \label{thm:general-properties}
\begin{enumerate}[label=(\roman*)]
\item A hom 1-bounded triangulated category is cone finite.
\item A cone finite triangulated category with a classical generator is countable.
\end{enumerate}
\end{theorem}

\begin{proof}
(i)
Suppose that $\sD$ is hom bounded with bound 1, i.e.\ $\hom(A,B) \leq 1$ for all $A,B \in \ind{\sD}$. In particular, this implies that nonzero morphisms $A \to B$ with $A,B\in\ind{\sD}$ have isomorphic cones. Consider a morphism of the form
\[
    A_1 \oplus \cdots \oplus A_n \xrightarrow{(a_1,\ldots,a_n)^t} B
\]
where $A_1, \dots, A_n$ and $B$ are indecomposable. 
If one $a_i=0$, then a standard application of the octahedral axiom shows that the cone splits up as follows:
\[
    A_1 \oplus \cdots \oplus A_n \xrightarrow{(a_1,\ldots,a_n)^t} B \too C((a_1,\dots, a_{i-1}, a_{i+1}, \dots a_n)^t) \oplus \Sigma A_i .
%    A_1 \oplus A_2 \xrightarrow{(a_1 ~ 0)^t} B \too C(a_1) \oplus \Sigma A_2 .
\]
Therefore, we can assume that all $a_i\neq0$.
Any other such morphism $A_1\oplus\cdots\oplus A_n$ is of the form $(\lambda_1 a_1,\ldots,\lambda_n a_n)^t$ for scalars $\lambda_1,\ldots,\lambda_n\in\kk$ and hence induces a commutative diagram of distingished triangles
\[ \xymatrix{
  A_1 \oplus \cdots \oplus A_n \ar[rr]^-{(\lambda_1 a_1, \ldots, \lambda_n a_n)^t}
                               \ar[d]_-{\text{diag}(\lambda_1,\ldots,\lambda_n)}
& & B  \ar[r] \ar@{=}[d]
& C    \ar[r]  \ar@{-->}[d]
& \Sigma (A_1 \oplus \cdots \oplus A_n) \ar[d]
\\
  A_1 \oplus \cdots \oplus A_n \ar[rr]_-{(a_1, \ldots, a_n)^t}  & & B \ar[r]   & C' \ar[r]
& \Sigma (A_1 \oplus \cdots \oplus A_n) .
} \]
As above, we may assume that all $\lambda_i \neq 0$. But then both vertical morphisms in the left hand square are isomorphisms, hence the dashed arrow is also an isomorphism. It follows that there are finitely many  possible cones $C$ for morphisms $A_1 \oplus \cdots \oplus A_n \to B$.

Now consider the cone of an arbitrary morphism $\bigoplus_{i=1}^n A_i \to \bigoplus_{j=1}^m B_j$. We proceed by induction on $m$. For $m=1$, we are done above, so assume $m > 1$. We have the following diagram coming from the octahedral axiom:
\[
\xymatrix{
                                                            & B_1 \ar@{=}[r]  \ar[d]                     & B_1 \ar[d] \\
\bigoplus_{i=1}^n A_i \ar[r] \ar@{=}[d] & \bigoplus_{j=1}^m B_j \ar[r] \ar[d] & C \ar[d] \\
\bigoplus_{i=1}^n A_i \ar[r]                  & \bigoplus_{j=2}^m B_j \ar[r]           & C'.
}
\]
By induction, there are finitely many possible $C'$. By the case for $m=1$ above, for each $C'$ there are finitely many possible cones of $C' \to \Sigma B_1$, and in particular, finitely many possibilities for $C$. Hence $\sD$ is cone finite.

(ii)
Recall that a \emph{classical generator} of a triangulated category $\sD$ is an object $G$ such that every object is obtained from $G$ in finitely many steps by taking shifts, cones and summands. (There does not have to be a bound on the number of steps; if such bounds exist, their minimum is the Rouquier dimension of $\sD$.) Bounded derived categories of finite-dimensional algebras and projective varieties have classical generators.

In each step, there are finitely many possibilities for summands by Krull--Schmidt; countably many possibilities for sums; and finitely many possibilities for cones by the  assumption. Therefore, with $G$ generating $\sD$ in countably many steps, the cardinality of objects of $\sD$ is countable as well.
\end{proof}

\begin{remark} \label{rem:fields}
\begin{enumerate}
\item The assumption of a classical generator in the theorem cannot be dropped: if $\sD$ is any cone finite triangulated category, and $I$ some uncountable set, then $\bigoplus_I\sD$ is still cone finite but uncountable.
\item The same proof shows a bit more: if $\sD$ has a classical generator and for all $A,B\in\sD$, there are only countably many cones of morphisms $A\to B$, then $\sD$ is countable. In particular, this applies to any Hom-finite $\sD$ if $\kk$ is a countable field.
\item We remark that if $\kk$ is a finite field, then $\sD$ Hom-finite trivially implies cone finite. Hence, for fields of arbitrary cardinality, hom boundedness captures `smallness' best among the notions of \hyref{Definition}{def:triangulated-notions}. See the table in \hyref{Section}{sec:examples}.
\end{enumerate}
\end{remark}

\begin{conjecture}
Hom bounded triangulated categories are cone finite.
\end{conjecture}

\section{Discreteness with respect to t-structures} \label{sec:Vossieck}
\noindent
In this section we consider an abstracted version of Vossieck's \cite{Vossieck} original definition of derived-discrete algebras. Examining the derived categories of these algebras in \cite{BPP} was our motivation to introduce the categorical notions in this article.

A \emph{torsion pair} in a triangulated category $\sD$ consists of full subcategories $(\sX,\sY)$, each of which is closed under taking direct summands, such that $\Hom(\sX,\sY) = 0$ and 
\[ \sD = \sX * \sY \coloneqq \{ D \in \sD \mid \exists \, \tri{X}{D}{Y} \text{ with } X \in \sX, Y \in \sY\}.\] 
A torsion pair $(\sX,\sY)$ is \emph{bounded}\footnote{%
  Be aware that this a different use of the  word `bounded' than the one from
  \hyref{Definition}{def:triangulated-notions}(2), the property `hom bounded'
  of triangulated categories. No confusion should arise.}
if $\bigcup_{n\in \IZ} \Sigma^n \sX = \sD = \bigcup_{n \in \IZ} \Sigma^n \sY$.

A \emph{t-structure} is a torsion pair $(\sX,\sY)$ with $\Sigma \sX \subseteq \sX$.
Any t-structure induces an abelian category, its \emph{heart} $\sH \coloneqq \sX \cap \Sigma \sY$.
To any t-structure in a triangulated category, there are associated cohomology functors, denoted $H^i\colon\sD\to\sH$ for $i\in\IZ$.

Bounded t-structures can be reconstructed from their hearts. In this note, we only deal with bounded t-structures, and we will simply say `bounded heart' to mean a full abelian subcategory which is the heart of a bounded t-structure. 
Background on t-structures can be found in, for example, \cite[\S 10.1]{KS}.

\begin{definition}
Let $\sD$ be a Hom-finite, Krull--Schmidt $\kk$-linear triangulated category admitting bounded t-structures.
\begin{enumerate}[label=(\arabic*)]
\item Let $\sH$ be the heart of a bounded t-structure. Then $\sD$ is said to be \emph{discrete with respect to $\sH$}, or \emph{$\sH$-discrete},
      if for every group valued function $v\colon\IZ\to K_0(\sD)$, the set of objects
       $\{D\in\sD \mid [H^i(D)] = v(i) \in K_0(\sD) ~\forall i\in\IZ\}$
      is finite.
\item $\sD$ is said to have \emph{finite hearts} if the heart of any bounded t-structure in $\sD$ has only finitely many
      indecomposable objects.
\end{enumerate}
\end{definition}

In (1), $v$ is not assumed to be a group homomorphism.

There is a well-known, canonical isomorphism of Grothendieck groups, $K_0(\sH) = K_0(\sD)$, induced by the inclusion $\sH\into\sD$. The inverse is given by sending the class of a complex to the alternating sum of the classes of its cohomologies with respect to $\sH$.
Moreover, given a finite-dimensional $\kk$-algebra $\Lambda$ having $N$ (non-isomorphic) simple modules, there is a further canonical isomorphism $K_0(\Lambda) \coloneqq K_0(\mod{\Lambda}) = \IZ^N$, mapping the class of a module to its dimension vector.

Let us establish the link between the above definition of discreteness and Vossieck's original notion: in \cite{Vossieck}, he exclusively considers categories of the form $\sD = \Db(\Lambda)$ for finite-dimensional $\kk$-algebras over an algebraically closed field $\kk$. He calls the derived category $\Db(\Lambda)$ of the algebra $\Lambda$ \emph{discrete} if for any sequence $v\colon\IZ\to K_0(\Lambda)$ with only finitely many nonzero terms, the set of isomorphism classes of indecomposable complexes $A\in\Db(\Lambda)$ with dimension vector $\dimv{A} \coloneqq (\dimv{H^i(A)})_{i\in\IZ} = v$ is finite. Using the standard isomorphisms $K_0(\Db(\Lambda)) = K_0(\Lambda) = \IZ^N$, it is clear that $\Db(\Lambda)$ is discrete in Vossieck's sense if and only if $\Db(\Lambda)$ is discrete with respect to $\mod{\Lambda}$ in the above sense.

The finite hearts property came up in our previous work on spaces of stability conditions of derived-discrete algebras \cite{BPP2}.

Because we have to work with hearts in triangulated categories, we now also introduce some notions that capture `smallness' of abelian categories. Let $\sH$ be a Hom-finite, $\kk$-linear abelian category, then $\sH$ is Krull--Schmidt \cite{Atiyah}. We denote $\pi\colon \Obj(\sH) \to K_0(\sH)$. Recall that we identify objects up to isomorphism. For an object $H\in\sH$, we denote by $\Subs(H)$ the set of subobjects $H'\into H$, and by $\Fact(H)$ the set of factors $H\onto H'$.

\begin{itemize}
\item $\sH$ is a \emph{length category} if it is  artinian and noetherian.
\item $\sH$ is \emph{finite} if the set $\ind{\sH}$ is finite.
\item $\sH$ is \emph{(abelian) discrete} if $\pi$ has finite fibres,
      i.e.\ for any $c\in K_0(\sH)$, the set of objects
      $\pi^{-1}(c) = \{ A \in \sH \mid [A] = c \}$ is finite.
\item $\sH$ is \emph{modular} if $\pi(\Subs(H))$ is a finite
      set for all $H\in\sH$, i.e.\
      $\{ [H'] \mid \exists H' \into H \} \subseteq K_0(\sH)$ is finite.
\end{itemize}

Of these, length and modular are mild restrictions. For example, they hold for $\mod{\Lambda}$ with $\Lambda$ a finite-dimensional algebra. The other two conditions (discrete and finite) are severe restrictions.

\begin{remark} \label{rem:unique-modules}
% It is *not* true that for a representation-finite algebra, all its indecomposable
% modules are either preprojective or preinjective is false.
% The family $\Lambda(1,n,0)$ gives a counterexample, each of the simple objects is
% in a tau-orbit containing no indecomposable projective or injective module.
% This is responsible for the failure of the stronger property in these examples.
%
All hearts for derived-discrete algebras are representation-finite module categories by \cite[\S 7.1]{BPP}. Since, by an exercise in string combinatorics, the derived-discrete algebras not derived equivalent to $\Lambda(1,n,m)$ are also representation-directed (see \cite[Ch.~IX]{ASS}), they therefore satisfy a stronger property than abelian discrete: the class of an indecomposable module determines the module uniquely.
\end{remark}

We will justify the terminology `modular' below in
\hyref{Remark}{rem:dimension_vectors}.
For now, just observe that the condition is equivalent to the finiteness of $\pi(\Fact(H))$, since $[H'']=[H]-[H']$ for any short exact sequence $0\to H'\to H\to H''\to0$.

\begin{lemma} \label{lem:dimension-vectors}
Let $H', H''$ be objects of a modular abelian category $\sH$. Then the $K_0$-classes of objects $H$ with exact sequences $H'\to H\to H''$ are finitely determined by the classes $[H']$ and $[H'']$, i.e.\ the set
 $\{ [H] \mid \exists \, H' \to H \to H'' \text{ exact} \} \subseteq K_0(\sH)$ is finite.
\end{lemma}

\begin{proof}
The exact sequence $H' \xrightarrow{f} H \xrightarrow{g} H''$ leads to short exact sequences
\[ 0 \to \kernel(f) \to H' \to \image(f) \to 0 \qquad\text{and}\qquad 0 \to \image(g) \to H'' \to \coker(g) \to 0 . \]
As $\sH$ is modular, $[\image(f)]$ is finitely determined by $[H']$, and $[\image(g)]$ is finitely determined by $[H'']$.
Moreover, 
 $[H] = [\kernel(g)] + [\image(g)] = [\image(f)] + [\image(g)]$,
from $0 \to \kernel(g) \to H \to \image(g) \to 0$.
Thus, $[H]$ is finitely determined by $[H']$ and $[H'']$, as claimed.
\end{proof}

\begin{remark} \label{rem:dimension_vectors}
The property of the lemma captures the positivity of dimension vectors for modules over a finite-dimensional algebra $\Lambda$. If $\Lambda$ has $N$ simple modules, then $K_0(\Lambda)=\IZ^N$, and the class of a module $M$ is encoded in its dimension vector $\dimv{M}\in\IN^N$.

All submodules have smaller dimension vectors, hence $\sH\coloneqq\mod{\Lambda}$ is a modular abelian category. The lemma generalises the inequality $\dimv{M}\leq\dimv{M'}+\dimv{M''}$ for an exact sequence $M'\to M\to M''$.

Moreover, over an algebraically closed field $\kk$ the equivalence
\[
\mod{\Lambda} \text{ discrete} \iff \mod{\Lambda} \text{ finite, i.e.\ $\Lambda$ has finite representation type }
\]
holds by the validity of the second Brauer--Thrall conjecture; see for example \cite[Ch.~IV.5]{ASS} and the references therein.
In general, these notions are not equivalent. For example, tubes are Hom-finite hereditary abelian categories that are discrete but not finite.
\end{remark}

\begin{theorem} \label{thm:discrete-characterisations}
Let $(\sD,\sH)$ be a triangulated category together with the heart of a bounded t-structure. Then
\begin{enumerate}[label=\textnormal{(\roman*)}]
\item $\sD$ cone finite and $\sH$ abelian discrete $\implies$ $\sD$ is discrete with respect to $\sH$.
\item $\sD$ is discrete with respect to $\sH$ $\implies$ $\sH$ abelian discrete.
\item $\sD$ is discrete with respect to $\sH$ and $\sH$ modular $\implies$ $\sD$ cone finite.
\end{enumerate}
\end{theorem}

\begin{corollary}
Let $(\sD,\sH)$ be a triangulated category with a modular heart of a bounded t-structure. Then
\[ \sD \text{ is discrete with respect to } \sH \iff \sD \text{ cone finite and } \sH \text{ abelian discrete} . \]
\end{corollary}

\begin{proof}
For an object $D\in\sD$, we define the function $v_D\colon \IZ\to K_0(\sD)$ by $v_D(i) \coloneqq [H^i(D)]$.

For a function $v\colon \IZ\to K_0(\sD)$, we define:
\begin{itemize}
\item $\sD_v \coloneqq \{ A \in \sD \mid [H^i(A)] = v(i) ~\forall i\in\IZ \}$, a full subcategory;
\item $\supp(v) \coloneqq \{ i \in \IZ \mid v(i) \neq 0 \}$, the support of $v$;
\item $\length{v} \coloneqq \max\limits_{i\in\IZ}\{v(i)\neq0\} - \min\limits_{i\in\IZ}\{v(i)\neq0\}$, the length of $v$.
\end{itemize}

\noindent
(i)
Given $v$, we do induction on the length of $v$. If $\length{v}=0$, then we can assume that $v(0)\neq0$, by suspending if necessary. Then all objects of $\sD_v$ have a single cohomology in degree 0, hence are in the heart $\sH$. Thus,
 $\sD_v = \{ A \in \sH \mid [A] = v(0) \}$,
and this set is finite by our assumption that $\sH$ is abelian discrete.

Now let $\length{v}=n>0$. Again, without loss of generality, we can assume that $\supp(v) \subseteq \{0,\ldots,n\}$. Define
$v', v'' \colon \IZ\to K_0(\sD)$ by $v''(i) = v(i)$ for $1\leq i\leq n$ and zero otherwise; and $v'(0) = v(0)$ and zero otherwise.
By induction, the subcategories $\sD_{v'}$ and $\sD_{v''}$ are finite. Now for any object $A\in\sD_v$, the truncation triangle for $A$ with respect to $H^0$ has the form
 $A' \to A \to A'' \to \Sigma A'$
with $A'\in\sD_{v'}$ and $A''\in\sD_{v''}$.
Hence, $\sD_v \subseteq \sD_{v'} * \sD_{v''}$. However, as $\sD$ is cone finite, there are only finitely many cones out of the finitely many objects from the two subcategories. Hence, $\sD_v$ is also finite.

\medskip \noindent
(ii)
This is immediate: given $c\in K_0(\sH)$, define $v\colon\IZ\to K_0(\sD)=K_0(\sH)$ by $v(0) \coloneqq c$ and $v(i)=0$ for $i\neq0$. Since $\sD$ is discrete with respect to $\sH$, the set of objects $D\in\sD$ with $[H^i(D)] = v(i)$ for all $i\in\IZ$ is finite. By construction, $H^i(D)=0$ for all $i\neq0$, i.e.\ $D\in\sH$, and $[D]=v(0)=c\in K_0(\sH)$. Hence $\sH$ is abelian discrete.

\medskip \noindent
(iii)
For $A,B\in\sD$, we want to show that there are only finitely many cones $A\xrightarrow{f} B\to C_f$, where $f\in\Hom(A,B)$ is arbitrary. Any such triangle gives rise to a long exact cohomology sequence in $\sH$
\[ \cdots \to H^i(A) \to H^i(B) \to H^i(C_f) \to H^{i+1}(A) \to H^{i+1}(B) \to \cdots . \]
By \hyref{Lemma}{lem:dimension-vectors}, $[H^i(C_f)]\in K_0(\sH)$ is determined up to finite ambiguity by $[H^i(B)]$ and $[H^{i+1}(A)]$. As the long exact sequence is finite (the t-structure is bounded), we see that all $[H^i(C_f)]$ are determined by $A$ and $B$, up to finite ambiguity (even more, they are determined by the functions $v_A,v_B$, but we do not need this). Hence, for fixed $A$ and $B$, there are only finitely many possibilities for $v_{C_f}$. Finally, since $\sD$ is $\sH$-discrete it follows that for each such choice of $v_{C_f}$, there are only finitely many objects $C_f$ realising this function. Altogether, the number of cones of morphisms $A\to B$ is finite.
\end{proof}

We expect the following statements to hold in general. In the next section, we show that they do hold for derived-discrete algebras. Note that a triangulated category $\sD$ can be discrete with respect to a bounded heart $\sH$ which is hom unbounded; see the tube category $\sphc{1}$ in the table on \hypageref{tab:examples}. This example also yields a bounded heart with infinitely many indecomposable objects.

\begin{conjecture} \label{conj:Vossieck}
Let $\sD$ be a Hom-finite Krull--Schmidt triangulated category and $\sH$ the heart of a bounded t-structure.
\begin{enumerate}[label=\textnormal{(\roman*)}]
\item If $\sD$ is $\sH$-discrete, then all bounded hearts in $\sD$ are discrete.
\item If $\sD$ is $\sH$-discrete, then so is $(\sD,\sH')$ for any bounded heart $\sH'$.
\item If $\sD$ is $\sH$-discrete and $\sH$ is finite, then all bounded hearts are finite.
\item $\sD$ is $\sH$-discrete $\iff$ $\sD$ is cone finite.
\end{enumerate}
\end{conjecture}

\section{Derived-discrete algebras $\LLambda$} \label{sec:DDA}

\noindent
Recall that in \cite{Vossieck}, a finite-dimensional algebra was defined to have a \emph{discrete derived category} if $\Db(\Lambda)$ is discrete with respect to $\mod{\Lambda}$ in our sense, i.e.\ with respect to the standard heart. Following standard usage, we call such an algebra \emph{derived-discrete}.

By the classification of G. Bobi\'nski, C. Gei\ss\ and A. Skowro\'nski \cite{BGS}, such an algebra is derived equivalent to either a representation-finite hereditary algebra or to the path algebra $\LLambda$ given by a cycle of length $n$ to which a linearly oriented $A_m$-chain is attached; bound by $r$ consecutive zero relations in the cycle, ending at the trivalent vertex. Here, $m\geq0, 1\leq r \leq n$. In the following, we assume that $r<n$, which is equivalent to $\LLambda$ having finite global dimension.

\medskip
\begin{center}
\includegraphics[width=0.45\textwidth]{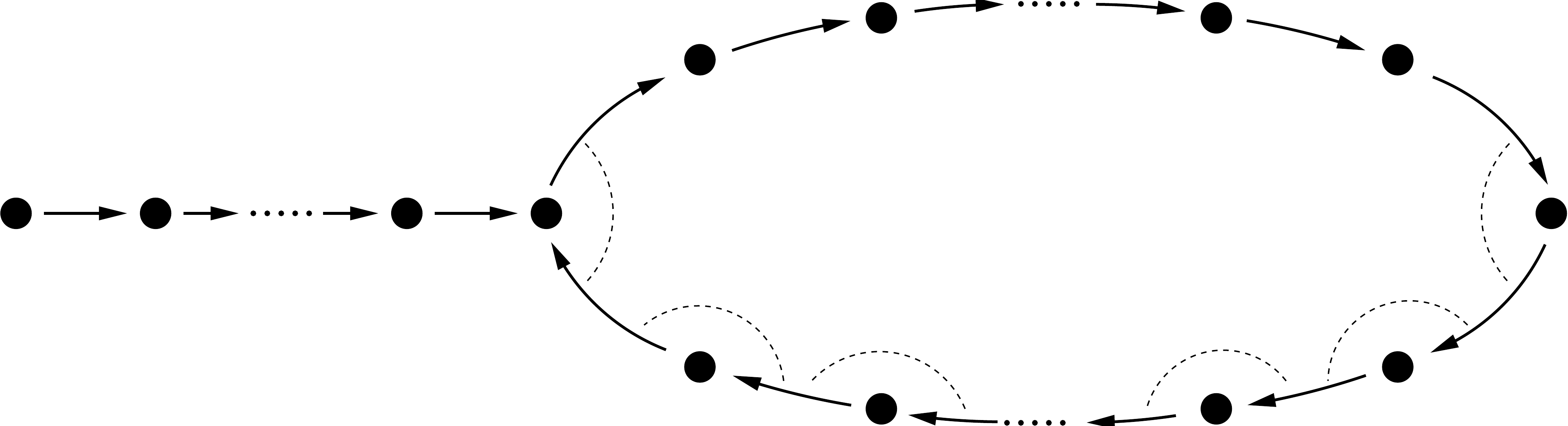}
\end{center}
\medskip

We mention two basic facts about these categories from our previous work:

\begin{proposition}[{\cite[Theorem~5.1 and Proposition~6.1]{BPP}}] \label{prop:BPP-results} $ $\newline
Let $\sD=\Db(\LLambda)$ be the bounded derived category of a derived-discrete algebra. Then
\begin{enumerate}[label=\textnormal{(\roman*)}]
\item If $r=1$, then $\sD$ is hom 2-bounded. If $r>1$, then $\sD$ is hom 1-bounded.
\item Hearts of bounded t-structures are finite.
\end{enumerate}
\end{proposition}

We use this together with the results of \hyref{Section}{sec:Vossieck} to show that discreteness for $\LLambda$ is actually independent of the bounded t-structure.

\begin{proposition} \label{prop:DDC-heart-independence}
Let $\Lambda$ be a derived-discrete algebra. Then $\Db(\Lambda)$ is discrete with respect to any bounded heart $\sH$.
\end{proposition}

\begin{proof}
As $\mod{\Lambda}$ is the module category of a finite-dimensional algebra (in particular modular) and $\Db(\Lambda)$ is $\mod{\Lambda}$-discrete, the triangulated category $\Db(\Lambda)$ is cone finite by \hyref{Theorem}{thm:discrete-characterisations}(iii). Moreover, any bounded heart $\sH$ is finite by \hyref{Proposition}{prop:BPP-results}. Therefore, $\Db(\Lambda)$ is discrete with respect to $\sH$ by \hyref{Theorem}{thm:discrete-characterisations}(i).
\end{proof}

\hyref{Propositions}{prop:BPP-results} and \ref{prop:DDC-heart-independence} show that \hyref{Conjecture}{conj:Vossieck} holds for $\Db(\LLambda)$.
When $r>1$, these categories even enjoy a property slightly stronger than classical Vossieck-discreteness, i.e.\ $\sH$-discreteness with respect to the standard heart: classes of cohomology modules of indecomposable complexes are even unique. 
For this statement, we can relax the finite global dimension assumption and allow $n = r$.

\begin{proposition} \label{prop:Vossieck-unique}
If $\Lambda$ is a derived-discrete algebra with $\Db(\Lambda) \simeq \Db(\LLambda)$ with $n \geq r> 1$, then the indecomposable complexes $D$ in $\Db(\Lambda)$ are uniquely determined by the sequences $[H^i(D)] \in K_0(\sD)$.
\end{proposition}

% This statement holds for any derived-equivalent heart of $\LLambda$, not just the standard heart. However, it is not clear if it's true for a heart coming from a proper silting object.

\begin{proof}
We explain the computation for $\Lambda \coloneqq \LLambda$; the computation for a derived-equivalent presentation of $\Lambda$ is similar.
Since $\Lambda$ is gentle, by \cite{BM} the indecomposable complexes in $\Kbminus(\Lambda)$ are given by \emph{(homotopy) string complexes}; see \cite{Bobinski} for the terminology and \cite[\S 2]{ALP} for an overview. The string complexes for $\Lambda$ are listed in \cite[Lemma 7.1]{ALP}. Now \cite[Theorem 2.8]{CPS} describes the strings of the indecomposable summands of $H^\bullet(D)$ for a string complex $D$, which implies that $H^i(D)$ is either indecomposable or zero. 

Using the orientations of the homotopy strings given in \cite[\S 7]{ALP}, one sees that the cohomology of each string complex over $\Lambda$ has an indecomposable module in its minimal degree, a (possibly trivial) periodic part, corresponding to the cycle, consisting of a repeating unique indecomposable module, and an indecomposable module in its maximal degree. In the infinite global dimension case, the periodic part has zero cohomology. 
In particular, each string complex is uniquely specified by its cohomology.
The result then follows because each indecomposable $\Lambda$-module is uniquely determined by its dimension vector; see \hyref{Remark}{rem:unique-modules} and recall $r>1$.
\end{proof}

%\begin{example}
%The statement of \hyref{Proposition}{prop:Vossieck-unique} does not hold for all $\sH$-discrete categories. For instance, the derived category $\Db(\IP^1_\kk)$ for a finite field $\kk$ is $\sH$-discrete by \hyref{Remark}{rem:fields} but the finitely many skyscraper sheaves have the same class in the Grothendieck group.
%\end{example}

\begin{example}
In addition, discrete derived categories are not `cone unique': there are indecomposable objects $A,B$ of $\Db(\Lambda(1,2,1))$ and nonzero maps $f,g\colon A\to B$ having non-isomorphic cones. The pathology again only occurs in the case $r=1$; obviously such behaviour is impossible when $r>1$, as then Hom spaces between indecomposable complexes are 1-dimensional and lead to unique cones.

Consider $\Lambda(1,2,1)$, i.e.\ the path algebra for the quiver
  \smash{$\xymatrix{-1 \ar[r]|a & 0 \ar@<0.7ex>[r]|b & 1 \ar@<0.5ex>[l]|c}$}
bound by the zero relation $bc$ at the vertex 0. Write $P \coloneqq P(-1)$ and $Q \coloneqq P(0)$ for the projective modules associated to vertices $-1$ and 0, respectively.
\[
\parbox{0.45\textwidth}{
\xymatrix{
A= \ar@<-1ex>[d]_f   & Q          \ar[r]^-{(cb~a)}  \ar[d]_1           & Q \oplus P \ar[d]^{\binom{a}{0}} \\
B=                   & Q          \ar[r]^{cba}      \ar[d]_-{(cb ~ a)} & P          \ar[d]^1 \\
                     & Q \oplus P \ar[r]^-{\binom{a}{0}}               & P
}  }
\parbox{0.1\textwidth}{~~~}
\parbox{0.45\textwidth}{
\xymatrix{
                   & Q \ar[d]_{cb}        \ar[r]^-{(cb~a)}         & Q \oplus P \ar[d]^{\binom{0}{0}} & =A \ar@<1ex>[d]_g \\
                   & Q \ar[d]            \ar[r]^{cba}              & P          \ar[d]             & =B \\
Q \ar[r]^-{(a ~ cb ~ 0)} & P \oplus Q \oplus Q \ar[r]^-{(0 ~ cba ~ 0)^t} & P
}  }
\]
Then the cones are $(Q \to 0) \oplus P \to Q$ and
  $(0 \to Q \to 0) \oplus (P \to P \oplus Q \to Q)$.
Above, all the matrices are transposed because, to match up with string combinatorics for these examples, we read compositions of maps from left to right.
\end{example}

\begin{remark}
We expect that all results of this section are also true for the derived-discrete algebras $\Lambda(n,n,m)$ of infinite global dimension. In fact, the hom bound was established in \cite[Theorem 7.4]{ALP}.
\end{remark}

Our results suggest the following question.

\begin{question}
Are derived-discrete algebras $\Lambda$ characterised, among finite-dimensional algebras, by $\Db(\Lambda)$ having only finite bounded hearts?
\end{question}

\subsubsection*{Relation to compactly generated triangulated categories}
We briefly discuss `big' triangulated categories, i.e.\ assuming the existence of all set-indexed coproducts. Such categories are not Hom-finite, and hence outside the scope of the rest of this article.

Recall that a (compactly generated) triangulated category $\sD$ is called \emph{pure-semisimple} if each object $D \in \sD$ is pure-injective; we refer the reader to \cite[Section 2]{Garkusha-Prest} for the definition of pure-injectivity in the setting of triangulated categories.

Let $\Lambda$ be a finite-dimensional algebra.
By \cite[Theorem 12.20]{Beligiannis}, the category $\sD = \sD(\Mod{\Lambda})$ is pure-semisimple if and only if $\Lambda$ is derived equivalent to a representation-finite hereditary algebra. However, by \cite{ALPP}, each indecomposable object in the homotopy category $\sK(\Proj{\LLambda})$ is pure-injective but $\sK(\Proj{\LLambda})$ is not pure-semisimple; similarly also for $\sD(\Mod{\LLambda})$.
This raises the following question:

\begin{question}
Does the property that each indecomposable object of a big compactly generated triangulated category is pure-injective, characterise discreteness among big triangulated categories?
\end{question}

\section{Discreteness with respect to co-t-structures} \label{sec:co-discrete}
\noindent
A \emph{co-t-structure} in a triangulated category $\sD$ is a torsion pair $(\sX,\sY)$ such that $\Sigma^{-1} \sX \subseteq \sX$. Its \emph{co-heart} is defined to be $\sC \coloneqq \sX \cap \Sigma^{-1} \sY$, which is an additive category that is, in general, not abelian.

Let $(\sD,\sC)$ be a Krull--Schmidt triangulated category with the co-heart $\sC$ of a bounded co-t-structure. 
By \cite[Corollary 5.9]{MSSS}, there is a natural bijection between bounded co-hearts of $\sD$ and silting subcategories $\sC$, i.e.\ $\Hom^{>0}(\sC,\sC))=0$ and $\thick{\sD}{\sC}=\sD$.
If $\sC$ has an additive generator, i.e.\ $\sC = \add C$ for some $C \in \sD$, then $C$ is a \emph{silting object}.

By \cite{MSSS}, see also \cite{AI}, the bounded co-t-structure $(\sX_\sC,\sY_\sC)$ can be recovered from the co-heart $\sC$ using the formulas
\[
\sX_\sC = \bigcup_{k > 0} \Sigma^{-k} \sC * \Sigma^{-k+1} \sC * \cdots * \Sigma^{-1} \sC
\quad \text{and} \quad
\sY_\sC = \bigcup_{k \geq 0} \sC * \Sigma \sC * \cdots * \Sigma^k \sC \, .
\]
For integers $p \leq q$ we set $\sC^{p,q} \coloneqq \Sigma^p \sC * \Sigma^{p+1} \sC * \cdots * \Sigma^{q-1} \sC * \Sigma^q \sC$. 

Let $D \in \sD$. If $D \in \Sigma^m \sC^{p,q}$ for some $m, p, q \in \IZ$ we shall say that $D$ is \emph{$(q-p)$-term with respect to $\sC$}.

The above formulas for $\sX_\sC$ and $\sY_\sC$ boil down to the observation that given a bounded co-t-structure with co-heart $\sC$, each object $0 \neq D \in \sD$ admits a Postnikov tower
\begin{equation} \label{eq:Postnikov}
\xymatrix{
0 = D_0 \ar[r] & D_1 \ar[r] \ar[d]            & D_2 \ar[r] \ar[d]            & \cdots \ar[r] & D_{n-1} \ar[r] \ar[d]  & D_n = D \ar[d] \\
               & \Sigma^{i_1} C_1 \ar@{-->}[ul] & \Sigma^{i_2} C_2 \ar@{-->}[ul] &               & \Sigma^{i_{n-1}} C_{n-1} & \Sigma^{i_n} C_n \ar@{-->}[ul]
}
\end{equation}
with $i_1 < i_2 < \cdots < i_n$ and $C_i \in \sC$; see \cite[Proposition 1.5.6]{Bondarko}.

In analogy with bounded t-structures, there is a canonical isomorphism of groups $K_0^{\mathrm{split}}(\sC) = K_0(\sD)$, induced by the inclusion $\sC \into \sD$; see \cite[Theorem 5.3.1]{Bondarko}. Note that unlike for bounded hearts, we use the \emph{split} Grothendieck group of the co-heart.

\begin{definition}
Let $(\sD,\sC)$ be a triangulated category with a co-heart $\sC$ of a bounded co-t-structure. We call $(\sD, \sC)$ \emph{discrete with respect to $\sC$}, or \emph{$\sC$-discrete}, if for each group valued function $v \colon \IZ \to K_0(\sD)$ the following set of objects is finite:
\[
\{ D \in \sD \mid D \text{ admits a filtration \eqref{eq:Postnikov} such that } \supp(v) =\{i_1, \dots ,i_n \} \text{ and } [C_j] = v(i_j)\} \,.
\]
\end{definition}

In contrast to the implications of \hyref{Theorem}{thm:discrete-characterisations}, the following co-t-structure analogue gives an equivalence. To assess this, think of silting subcategories as having the modularity condition ``built in''; a concrete instance is the correspondence between silting subcategories and algebraic t-structures for finite-dimensional algebras in \cite{KY}.

\begin{theorem} \label{thm:co-discrete}
Let $(\sD,\sC)$ be a triangulated category together with the co-heart $\sC = \add C$ of a bounded co-t-structure. Then
\[
\sD \text{ is cone-finite} \iff \sD \text{ is discrete with respect to $\sC$}.
\]
\end{theorem}

\begin{proof}
$(\Longrightarrow)$ The proof is essentially the same as the proof of \hyref{Theorem}{thm:discrete-characterisations}(i), where we instead write for a function $v \colon \IZ \to K_0(\sD)$,
\[
\sD_v \coloneqq \{ A \in \sD \mid A \text{ admits a filtration \eqref{eq:Postnikov} with } \supp(v) =\{i_1, \dots ,i_n \} \text{ and } [C_j] = v(i_j)% \, \forall i_j \in \IZ 
\}.
\]
The only part where the proof differs is the base step of the induction, i.e.\ $\length{v} = 0$. Again, without loss of generality we may assume $v(0) \neq 0$. Since $\sC = \add{C}$, where $C = C_1 \oplus \cdots \oplus C_n$ say, is a silting object, each object $C' \in \sC$ decomposes uniquely as $C' = C_1^{m_1} \oplus C_2^{m_2} \oplus \cdots \oplus C_n^{m_n}$, whence $[C'] = m_1 [C_1] + \cdots + m_n [C_n]$. Therefore, the class of $[C'] \in K_0^{\mathrm{split}}(\sC)$ is uniquely determined by its Krull--Schmidt decomposition. This says, in particular, that $\sD_v$ is a singleton when $\length{v} = 0$. The remainder of the proof proceeds as in Theorem~\ref{thm:discrete-characterisations}(i), noting that the uniqueness of the decomposition triangle $\tri{A'}{A}{A''}$ is not required for the proof to work.

$(\Longleftarrow)$ Let $A, B \in \sD$. We want to show that the set
  $\sZ \coloneqq \{Z \mid \exists \, A \xrightarrow{f} B \to Z \to \Sigma A\}$
is finite. We proceed in two steps.

\Step{1} \emph{$A$ is $1$-term with respect to $\sC$ and $B$ is $n$-term with respect to $\sC$, for some $n \geq 1$.}

\smallskip

Without loss of generality we may assume that $A = \Sigma^m C$ for some $C \in \sC$ and some $m \in \IZ$ and $B \in \sC^{0,n}$. In particular, this means that $B$ admits a filtration,
\[
\xymatrix{
0 = B_{-1} \ar[r] & B_0 \ar[r] \ar[d]            & B_1 \ar[r] \ar[d]            & \cdots \ar[r] & B_{n-1} \ar[r] \ar[d]  & B_n = B \ar[d] \\
               & C_0 \ar@{-->}[ul] & \Sigma C_1 \ar@{-->}[ul] &               & \Sigma^{n-1} C_{n-1} & \Sigma^n C_n \ar@{-->}[ul]
}
\]
with the $C_i \in \sC$, some of them possibly zero. We consider various possibilities for $m$.

If $m < 0$ then $\Hom(A,B) = 0$ and $\sZ$ is trivially finite.

If $m \geq n$, then we get the following filtration for any $Z \in \sZ$,
\[
\xymatrix{
0 = B_{-1} \ar[r] & B_0 \ar[r] \ar[d]            & B_1 \ar[r] \ar[d]            & \cdots \ar[r] & B_{n-1} \ar[r] \ar[d]  & B_n = B \ar[d] \ar[r] & Z \, , \ar[d] \\
               & C_0 \ar@{-->}[ul] & \Sigma C_1 \ar@{-->}[ul] &               & \Sigma^{n-1} C_{n-1} & \Sigma^n C_n \ar@{-->}[ul] & \Sigma^{m+1} C \ar@{-->}[ul]
}
\]
whence by $\sC$-discreteness there are only finitely many $Z$ admitting a filtration with these filtrands, making $\sZ$ finite.

If $0 \leq m < n$, we consider the diagram coming from the octahedral axiom.
\[
\xymatrix{
                         & A \ar@{=}[r] \ar[d] & \Sigma^m C \ar[d]^-{0}  \\
B_{n-1} \ar[r] \ar@{=}[d] & B \ar[r] \ar[d]     & \Sigma^n C_n \ar[d] \\
B_{n-1} \ar[r]            & Z \ar[r]            & X
}
\]
Thus, $X = \Sigma^n C_n \oplus \Sigma^{m+1} C$. If $m = n-1$ then we get the filtration:
\[
\xymatrix{
0 = B_{-1} \ar[r] & B_0 \ar[r] \ar[d]            & B_1 \ar[r] \ar[d]            & \cdots \ar[r] & B_{n-1} \ar[r] \ar[d]  & Z \, .\ar[d] \\
               & C_0 \ar@{-->}[ul] & \Sigma C_1 \ar@{-->}[ul] &               & \Sigma^{n-1} C_{n-1} & \Sigma^n (C_n \oplus C) \ar@{-->}[ul]
}
\]
Otherwise, using \cite[Lemmas 7.1 and 7.2]{JP} in sequence gives the filtration 
\[
\xymatrix{
0 = B_{-1} \ar[r] & B_0 \ar[r] \ar[d]            &  \cdots \ar[r]  & B'_{m+1} \ar[r] \ar[d]                      & \cdots \ar[r] & B'_{n-1} \ar[r] \ar[d]  & Z \, .\ar[d] \\
                 & C_0 \ar@{-->}[ul]              &                & \Sigma^{m+1} (C_{m+1} \oplus C) \ar@{-->}[ul] &               & \Sigma^{n-1} C_{n-1} & \Sigma^n C_n \ar@{-->}[ul]
}
\]
In either case, $\sC$-discreteness affirms the finiteness of $\sZ$.

\Step{2} \textit{Both $A$ and $B$ are of arbitrary length with respect to $\sC$.}

\smallskip

\noindent
We proceed by induction on the length of $A$ with respect to $\sC$; cf.\ proof of \hyref{Theorem}{thm:general-properties}. Suppose $A$ is $n$-term with respect to $\sC$. Then $A$ admits a decomposition $\tri{A'}{A}{A''}$ in which $A'$ is $(n-1)$-term and $A''$ is $1$-term with respect to $\sC$. Now consider the diagram coming from the octahedral axiom:
\[
\xymatrix{
                    & \Sigma A' \ar@{=}[r] \ar[d] & \Sigma A' \ar[d] \\
Z \ar[r] \ar@{=}[d] & \Sigma A \ar[r] \ar[d]      & \Sigma B \ar[d] \\
Z \ar[r]            & \Sigma A'' \ar[r]           & \Sigma X .
}
\]
By induction, there are finitely many possible $\Sigma X$, whence by Step 1, and the fact that $A''$ is $1$-term with respect to $\sC$, there are finitely many possible $Z$.
\end{proof}

\begin{corollary}
If $\sD$ is discrete with respect to a silting subcategory $\sC$ of $\sD$, then $\sD$ is discrete with respect to any other silting subcategory $\sC'$.
\end{corollary}

For the next corollary, we remark that $\Db(\mod{\Lambda})$ has a natural bounded t-structure, with heart $\mod{\Lambda}$, and $\Kb(\proj{\Lambda})$ has a natural co-t-structure, with co-heart $\proj{\Lambda}$. 
However, by \hyref{Proposition}{prop:DDC-heart-independence} and the previous corollary, the actual choices of heart and co-heart do not matter.

\begin{corollary} \label{cor:co-discrete}
Let $\Lambda$ be a finite-dimensional $\kk$-algebra. If $\Db(\mod{\Lambda})$ is discrete with respect to $\mod{\Lambda}$ then  $\Kb(\proj{\Lambda})$ is discrete with respect to $\proj{\Lambda}$.
\end{corollary}

\begin{proof}
$\Db(\mod{\Lambda})$ is cone finite by \hyref{Theorem}{thm:discrete-characterisations} as $\mod{\Lambda}$ is modular. Hence the full subcategory of $\Kb(\proj{\Lambda}) \subseteq \Db(\mod{\Lambda})$ is also cone finite. Whence \hyref{Theorem}{thm:co-discrete} implies $\Kb(\proj{\Lambda})$ is discrete with respect to $\proj{\Lambda}$.
\end{proof}

\begin{remark}
Vossieck's main result \cite[\S 2, Theorem]{Vossieck} asserts that $\Db(\mod{\Lambda})$ is discrete if and only if $\Kb(\proj{\Lambda})$ is discrete (taking homology with respect to $\mod{\Lambda}$). 

However, in general, there is no intrinsic definition of discreteness in $\Kb(\proj{\Lambda})$. For example, $\Db(\mod{\kk[x]/(x^2)})$ is discrete with respect to $\mod{\Lambda}$, but $\Kb(\proj{\kk[x]/(x^2)})$ has no bounded t-structure \cite{HJY} so that the discreteness notion of \cite{Vossieck} does not apply. Nevertheless, $\Kb(\proj{\kk[x]/(x^2)}) = \Kb(\proj{\Lambda(1,1,0)})$ is discrete with respect to $\proj{\Lambda}$ by \hyref{Corollary}{cor:co-discrete}, and more generally, the same holds for all $\Kb(\proj{\Lambda(n,n,m)})$.
\end{remark}

\section{Examples} \label{sec:examples}
\noindent
In the following table, we present some triangulated categories exhibiting interesting behaviour with regards to the various smallness notions studied in this article. For the convenience of the reader, we briefly summarise these notions:
\begin{description}[font=\itshape]
\item[$\sH$-discrete]
     There is a bounded heart $\sH$ such that for any $v\colon \IZ\to K_0(\sH)$,
     the set of objects $D\in\sD$ with $[H^i(D)]=v(i)$ for all $i$ is finite.
     In all example classes below, if this property holds for one bounded heart,
     it holds for all.
\item[$\sC$-discrete]
     There is a silting subcategory $\sC$ such that for any $v\colon \IZ \to K_0^{\mathrm{split}}(\sC)$,
     the set of objects admitting a Postnikov tower having filtrands $\Sigma^{i_j}C_j$
     with $i_1<\cdots<i_n$ and $C_i\in\sC$ is finite. 
\item[hom bound]
     There is a universal bound on Hom dimensions among indecomposable objects;
     the subscript indicates the maximal bound occurring in the family.
\item[cone finite]
     Any two objects admit only finitely many cones, up to isomorphism.
\item[finite hearts]
     Hearts of bounded t-structures have finitely many indecomposables.
\item[discrete hearts]
     Any object $H\in\sH$ of any bounded heart is determined up to finite ambiguity
     by $[H]\in K_0(\sH)$.
\item[countable]
     The category has only countably many objects, up to isomorphism.
\end{description}
Several of these properties make no sense for triangulated categories without bounded (co-)t-structures. This is indicated by $-$ in the table. The examples assume that $\kk$ is an uncountable field.

\begin{center} \label{tab:examples}
\resizebox{\textwidth}{!}{
\begin{tabular}{ >{\small} l <{\normalsize} *{11}{p{2.25em}} } \toprule
     & & & & & & & & & \\[-2.7ex]

     & \mcc{$\kk\Gamma_{\!\scalebox{0.6}{\text{ADE}}}$} & \mcc{$\kk A_\infty$}
     & \mcc{DDC} & \mcc{$\text{DDC}^c$}
     & \mcc{$\orbit$}
     & \mcc{$\sphc{1}$} & \mcc{$\sphc{>1}$} & \mcc{$\sphc{<1}$} & \mcc{$\sphc{1,n}$}
     & \mcc{$\IQ\tilde A_1$} & \mcc{$\IFq\tilde A_1$} \\ \midrule
%
% examples:            ADE       A_infty   DDC       DDC^c     orbit     T_1    T_2       T_0      T_1,n     Q      IF_q
%
$\sH$-discrete    & \yes    & \yes    & \yes    & \na     & \na     & \yes & \yes    & \na     & \na     & \no  & \yes \\
$\sC$-discrete & \yes    & \yes    & \yes    & \yes    & \na     & \na  & \na     & \yes    & \na     & \no  & \yes \\
hom bounded          & \yeb{6} & \yeb{1} & \yeb{2} & \yeb{2} & \yeb{6} & \no  & \yeb{1} & \yeb{2} & \yeb{\mathrlap{\floor{n/2}}}
                                                                                                           & \no  & \no  \\
cone finite          & \yes    & \yes    & \yes    & \yes    & \yes    & \yes & \yes    & \yes    & \yes    & \no  & \yes \\
finite hearts        & \yes    & \no     & \yes    & \na     & \na     & \no  & \yes    & \na     & \na     & \no  & \no  \\
discrete hearts      & \yes    & \yes    & \yes    & \na     & \na     & \yes & \yes    & \na     & \na     & \no  & \yes \\
countable objects    & \yes    & \yes    & \yes    & \yes    & \yes    & \yes & \yes    & \yes    & \yes    & \yes & \yes \\
\bottomrule
\end{tabular}
}
\end{center}

\bigskip

We proceed to explain the example classes.

\begin{description}
\item[$\kk\Gamma_{\!\scalebox{0.6}{\text{ADE}}}$, $\kk A_\infty$ --- quiver algebras]
By listing finite-dimensional algebras, we mean their bounded derived categories.

$\Gamma_{\!\text{ADE}}$ stands for an ADE quiver, so that the corresponding algebra is hereditary and representation-finite. The maximal hom bound of $6$ is achieved in type $E_8$.

$A_\infty$ stands for the (one-sided) infinite, zigzag oriented quiver of type $A$. This example is interesting because $\Db(\kk A_\infty)$ is hom bounded, but has infinite hearts; note it does not have a classical generator.
\item[DDC, $\text{DDC}^c$ --- derived-discrete algebras]
DDC is a shorthand for $\Db(\LLambda)$, the derived-discrete category for the algebra with $r$ consecutive relations in an $n$-cycle and a tail of length $m$; see \hyref{Section}{sec:DDA}. We assume $r<n$, so that AR triangles (i.e.\ a Serre functor) exist. Note that $\Db(\Lambda(n,n,m))$ has no bounded co-t-structures.

$\text{DDC}^c$ stands for $\Kb(\proj{\Lambda(n,n,m)})$, the bounded homotopy category of projective modules over a derived-discrete algebra of infinite global dimension. This is both the subcategory of perfect complexes of $\Db(\Lambda(n,n,m))$ and the the subcategory of compact objects in $\sD(\Lambda(n,n,m))$. These algebras are gentle, hence Gorenstein, so that $\text{DDC}^c$ has AR triangles \cite{GR}.
\item[$\orbit$ --- cluster categories]
$\orbit = \Db(\kk\Gamma_{\!\text{ADE}})/\Sigma^{-1}\tau$ stands for the cluster category of type ADE, where $\tau$ is the Auslander--Reiten translation. It is triangulated by \cite{Keller-orbit} and has finitely many indecomposables.
\item[$\sphc{1}, \sphc{>1}, \sphc{<1}$ --- spherical generators]
For $w\in\IZ$, let $\sphc{w}$ be the triangulated category generated by a $w$-spherical object, i.e.\ an object whose derived endomorphism algebra is $\kk\oplus\Sigma^{-w}\kk$. Note that $\sphc{1}$ is the bounded derived category of the hereditary standard homogeneous tube.
Hom bounds for $\sphc{<0}$ are 1, and for $\sphc{0}$ it is 2.

The categories $\sphc{<1}$ have no bounded t-structures \cite{HJY}, making it pointless to ask for $\sH$-discreteness or finite hearts. Likewise, $\sphc{\geq1}$ has no bounded co-t-structures.
\item[$\sphc{1,n}$ --- truncated tubes]
For $n>1$, we let $\sphc{1,n} = \stmod{\kk[x]/(x^n)}$ be the stable module category, e.g.\ $\kk[x]/(x^2) = \Lambda(1,1,0)$.
The AR quiver of $\sphc{1,n}$ is the following truncated homogeneous tube:
\[
\xymatrix{
X_1 \ar@/^/[r] & X_2 \ar@/^/[r] \ar@/^/[l] & X_3 \ar@/^/[r] \ar@/^/[l] & \cdots \ar@/^/[r] \ar@/^/[l] & X_{n-2} \ar@/^/[r] \ar@/^/[l] & X_{n-1} .\ar@/^/[l]
}
\]
The unique projective module $X_n$ does not occur because we have taken the stable category. The algebra is self\-injective, hence $\sphc{1,n}$ is triangulated. As $\sphc{1,n}$ has only $n-1$ many indecomposable objects and is Krull--Schmidt, it is cone finite. Moreover, from $\dim \underline{\Hom}(X_i,X_i) = \min \{i, n-i\}$, we see that arbitrary hom bounds can be attained. Note that $\sphc{1,n}$ has no bounded t-structures. % not even non-trivial torsion pairs
\item[Small fields]
$\tilde A_1$ is the Kronecker quiver, and the last two columns denote $\Db(\IQ\tilde A_1)$ and $\Db(\IF_q\tilde A_1)$, respectively. Instead of $\IQ$, any infinite countable field works. By \hyref{Remark}{rem:fields}, we could replace $\tilde A_1$ by any finite-dimensional algebra. We chose the Kronecker quiver because it is manifestly non-discrete for uncountable fields.
\end{description}

\addtocontents{toc}{\protect{\setcounter{tocdepth}{-1}}}

\enlargethispage{3ex}
\bigskip
\noindent
\resizebox{\textwidth}{!}{{Email: \texttt{nathan.broomhead@plymouth.ac.uk, d.pauksztello@lancaster.ac.uk, dploog@math.fu-berlin.de}}}

\smallskip
\noindent
\resizebox{\textwidth}{!}{%
\begin{tabular}{@{}l@{}}
Nathan Broomhead, Mathematical Sciences, Plymouth University, Drake Circus, Plymouth, PL4 8AA, United Kingdom \\
David Pauksztello, Department of Mathematics and Statistics, Lancaster University, Lancaster, LA1 4YF, United Kingdom \\
David Ploog, Arnimallee 3, Mathematisches Institut, Freie Universit\"at Berlin, 14195 Berlin, Germany
\end{tabular}
}

\end{document}